\documentclass[12pt]{amsart}
\usepackage{amscd,amssymb,latexsym}
\usepackage{fullpage}
\newcommand{\Mdef}[2]{\newcommand{#1}{\relax \ifmmode #2 \else $#2$\fi}}

\newcommand{\tensor}{\otimes}

\newcommand{\sdr}{\rtimes}

\Mdef{\bhom}{\mathbf{\hat{H}om}}

\Mdef{\Mod}{\mathrm{mod}}

\newcommand{\st}{\; | \;}

\newtheorem{thm}{Theorem}[section]
\newtheorem{lemma}[thm]{Lemma}
\newtheorem{prop}[thm]{Proposition}
\newtheorem{cor}[thm]{Corollary}

\theoremstyle{definition}

\newcommand{\qqed}{\qed \\[1ex]}
\renewenvironment{proof}[1][\hspace*{-.8ex}]{\noindent {\bf Proof #1:\;}}{\qqed}

\Mdef{\PH} {\Phi^H}
\Mdef{\PK} {\Phi^K}
\Mdef{\PL} {\Phi^L}
\Mdef{\PT} {\Phi^{\T}}

\Mdef{\ef}{E{\cF}_+}
\Mdef{\etf}{\widetilde{E}{\cF}}
\Mdef{\eg}{E{G}_+}
\Mdef{\etg}{\tilde{E}{G}}

\Mdef{\infl}{\mathrm{inf}}
\Mdef{\defl}{\mathrm{def}}
\Mdef{\res}{\mathrm{res}}
\Mdef{\ind}{\mathrm{ind}}
\Mdef{\coind}{\mathrm{coind}}

\Mdef{\univ}{\mathcal{U}}

\Mdef{\Fp}{\mathbb{F}_p}
\Mdef{\Zpinfty}{\Z /p^{\infty}}
\Mdef{\Zpadic}{\Z_p^{\wedge}}

\newcommand{\lra}{\longrightarrow}

\newcommand{\lr}[1]{\langle #1 \rangle}

\newcommand{\Gspectra}{\mbox{$G$-{\bf spectra}}}

\newcommand{\spec}{\mathrm{Spec}}

\Mdef{\we}{\mathbf{we}}
\Mdef{\fib}{\mathbf{fib}}
\Mdef{\cof}{\mathbf{cof}}
\Mdef{\BI}{\mathcal{BI}}

\newcommand{\colim}{\mathop{  \mathop{\mathrm {lim}} \limits_\rightarrow} \nolimits}

\Mdef{\B}{\mathbb{B}}
\Mdef{\C}{\mathbb{C}}
\Mdef{\D}{\mathbb{D}}
\Mdef{\E}{\mathbb{E}}
\Mdef{\T}{\mathbb{T}}
\Mdef{\F}{\mathbb{F}}
\Mdef{\G}{\mathbb{G}}
\Mdef{\I}{\mathbb{I}}
\Mdef{\N}{\mathbb{N}}
\Mdef{\Q}{\mathbb{Q}}
\Mdef{\R}{\mathbb{R}}
\Mdef{\bbS}{\mathbb{S}}
\Mdef{\Z}{\mathbb{Z}}

\Mdef{\bA}{\mathbb{A}}
\Mdef{\bB}{\mathbb{B}}
\Mdef{\bC}{\mathbb{C}}
\Mdef{\bD}{\mathbb{D}}
\Mdef{\bE}{\mathbb{E}}
\Mdef{\bF}{\mathbb{F}}
\Mdef{\bG}{\mathbb{G}}
\Mdef{\bH}{\mathbb{H}}
\Mdef{\bI}{\mathbb{I}}
\Mdef{\bJ}{\mathbb{J}}
\Mdef{\bK}{\mathbb{K}}
\Mdef{\bL}{\mathbb{L}}
\Mdef{\bM}{\mathbb{M}}
\Mdef{\bN}{\mathbb{N}}
\Mdef{\bO}{\mathbb{O}}
\Mdef{\bP}{\mathbb{P}}
\Mdef{\bQ}{\mathbb{Q}}
\Mdef{\bR}{\mathbb{R}}
\Mdef{\bS}{\mathbb{S}}
\Mdef{\bT}{\mathbb{T}}
\Mdef{\bU}{\mathbb{U}}
\Mdef{\bV}{\mathbb{V}}
\Mdef{\bW}{\mathbb{W}}
\Mdef{\bX}{\mathbb{X}}
\Mdef{\bY}{\mathbb{Y}}
\Mdef{\bZ}{\mathbb{Z}}

\Mdef{\cA}{\mathcal{A}}
\Mdef{\cB}{\mathcal{B}}
\Mdef{\cC}{\mathcal{C}}
\Mdef{\mcD}{\mathcal{D}} 
\Mdef{\cE}{\mathcal{E}}
\Mdef{\cF}{\mathcal{F}}
\Mdef{\cG}{\mathcal{G}}
\Mdef{\mcH}{\mathcal{H}} 
\Mdef{\cI}{\mathcal{I}}
\Mdef{\cJ}{\mathcal{J}}
\Mdef{\cK}{\mathcal{K}}
\Mdef{\mcL}{\mathcal{L}}

\Mdef{\cM}{\mathcal{M}}
\Mdef{\cN}{\mathcal{N}}
\Mdef{\cO}{\mathcal{O}}
\Mdef{\cP}{\mathcal{P}}
\Mdef{\cQ}{\mathcal{Q}}
\Mdef{\mcR}{\mathcal{R}}
\Mdef{\cS}{\mathcal{S}}
\Mdef{\cT}{\mathcal{T}}
\Mdef{\cU}{\mathcal{U}}
\Mdef{\cV}{\mathcal{V}}
\Mdef{\cW}{\mathcal{W}}
\Mdef{\cX}{\mathcal{X}}
\Mdef{\cY}{\mathcal{Y}}
\Mdef{\cZ}{\mathcal{Z}}

\Mdef{\ca}{\mathcal{a}}
\Mdef{\ct}{\mathcal{t}}

\Mdef{\At}{\tilde{A}}
\Mdef{\Bt}{\tilde{B}}
\Mdef{\Ct}{\tilde{C}}
\Mdef{\Et}{\tilde{E}}
\Mdef{\Ht}{\tilde{H}}
\Mdef{\Kt}{\tilde{K}}
\Mdef{\Lt}{\tilde{L}}
\Mdef{\Mt}{\tilde{M}}
\Mdef{\Nt}{\tilde{N}}
\Mdef{\Pt}{\tilde{P}}

\Mdef{\tA}{\tilde{A}}
\Mdef{\tB}{\tilde{B}}
\Mdef{\tC}{\tilde{C}}
\Mdef{\tE}{\tilde{E}}
\Mdef{\tH}{\tilde{H}}
\Mdef{\tK}{\tilde{K}}
\Mdef{\tL}{\tilde{L}}
\Mdef{\tM}{\tilde{M}}
\Mdef{\tN}{\tilde{N}}
\Mdef{\tP}{\tilde{P}}

\Mdef{\ft}{\tilde{f}}
\Mdef{\xt}{\tilde{x}}
\Mdef{\yt}{\tilde{y}}

\Mdef{\Ab}{\overline{A}}
\Mdef{\Bb}{\overline{B}}
\Mdef{\Cb}{\overline{C}}
\Mdef{\Db}{\overline{D}}
\Mdef{\Eb}{\overline{E}}
\Mdef{\Fb}{\overline{F}}
\Mdef{\Gb}{\overline{G}}
\Mdef{\Hb}{\overline{H}}
\Mdef{\Ib}{\overline{I}}
\Mdef{\Jb}{\overline{J}}
\Mdef{\Kb}{\overline{K}}
\Mdef{\Lb}{\overline{L}}
\Mdef{\Mb}{\overline{M}}
\Mdef{\Nb}{\overline{N}}
\Mdef{\Ob}{\overline{O}}
\Mdef{\Pb}{\overline{P}}
\Mdef{\Qb}{\overline{Q}}
\Mdef{\Rb}{\overline{R}}
\Mdef{\Sb}{\overline{S}}
\Mdef{\Tb}{\overline{T}}
\Mdef{\Ub}{\overline{U}}
\Mdef{\Vb}{\overline{V}}
\Mdef{\Wb}{\overline{W}}
\Mdef{\Xb}{\overline{X}}
\Mdef{\Yb}{\overline{Y}}
\Mdef{\Zb}{\overline{Z}}

\Mdef{\db}{\overline{d}}
\Mdef{\hb}{\overline{h}}
\Mdef{\qb}{\overline{q}}
\Mdef{\rb}{\overline{r}}
\Mdef{\tb}{\overline{t}}
\Mdef{\ub}{\overline{u}}
\Mdef{\vb}{\overline{v}}

\Mdef{\hc}{\hat{c}}
\Mdef{\he}{\hat{e}}
\Mdef{\hf}{\hat{f}}
\Mdef{\hA}{\hat{A}}
\Mdef{\hH}{\hat{H}}
\Mdef{\hJ}{\hat{J}}
\Mdef{\hM}{\hat{M}}
\Mdef{\hP}{\hat{P}}
\Mdef{\hQ}{\hat{Q}}

\Mdef{\thetab}{\overline{\theta}}
\Mdef{\phib}{\overline{\phi}}

\Mdef{\uA}{\underline{A}}
\Mdef{\uB}{\underline{B}}
\Mdef{\uC}{\underline{C}}
\Mdef{\uD}{\underline{D}}

\Mdef{\bolda}{\mathbf{a}}
\Mdef{\boldb}{\mathbf{b}}
\Mdef{\bfD}{\mathbf{D}}

\Mdef{\fm}{\frak{m}}
\Mdef{\fp}{\frak{p}}

\newcommand{\fX}{\mathfrak{X}}

\Mdef{\eps}{\epsilon}

\input{xypic}
\newcommand{\sub}{\mathrm{Sub}}

\newcommand{\cEi}{\cE^{-1}}
\newcommand{\cNbar}{\overline{\cN}}
\newcommand{\Hbar}{\overline{H}}

\newcommand{\modules}{\mbox{-modules}}

\newcommand{\bbN}{\mathbb{N}}
\newcommand{\Ah}{\hat{A}}
\renewcommand{\tb}{\overline{\times}}

\newcommand{\diag}{\mathrm{diag}}
\newcommand{\Tt}{\tilde{T}}

\begin{document}
\title{Rational $SU(3)$-equivariant cohomology theories}

\author{J.P.C.Greenlees}
\address{Mathematics Institute, Zeeman Building, Coventry CV4, 7AL, UK}
\email{john.greenlees@warwick.ac.uk}

\date{}

\begin{abstract}
We describe the spectral space of conjugacy classes of subgroups of
$SU(3)$, together with the additional structure of a sheaf of rings
and a component structure. It is a disjoint union of 18
blocks each dominated by a subgroup. For each of these blocks we identify a sheaf of rings
and component structure. Taken together, this gives an abelian
category $\cA (SU(3))$ designed to reflect the structure of rational $SU(3)$-equivariant cohomology theories, and we assemble the
results from elsewhere to show that the category of rational
$SU(3)$-spectra  is Quillen equivalent to the category of differential
graded objects of $\cA (SU(3))$.

\end{abstract}
\thanks{The author is grateful for comments, discussions  and related
  collaborations with S.Balchin, D.Barnes, T.Barthel, M.Kedziorek,
  L.Pol, J.Williamson. The work is partially supported by EPSRC Grant
  EP/W036320/1. The author  would also  like to thank the Isaac Newton
  Institute for Mathematical Sciences, Cambridge, for support and
  hospitality during the programme Equivariant Homotopy Theory in
  Context, where later parts of  work on this paper was undertaken. This work was supported by EPSRC grant EP/Z000580/1.  } 
\maketitle

\tableofcontents
\section{Introduction}
\subsection{Context}
It is conjectured \cite{AGconj} that for each compact Lie group $G$ there is an abelian
category $\cA (G)$ and a Quillen equivalence between the category of
rational $G$-spectra and the category of differential graded objects
of $\cA(G)$: 
$$\Gspectra \simeq_Q \mbox{DG-}\cA (G). $$
This is known for a range of small groups and the purpose of this note
is to prove the conjecture for $G=SU(3)$ and all its subgroups.

\subsection{Strategy}
In effect the algebraic model $\cA (G)$ is built by assembling data for each conjugacy
class of subgroups $H\subseteq G$. Indeed, the data over $(H)$ is a
model for $G$-spectra with geometric isotropy concentrated on the
single conjugacy class $(H)$, which is equivalent to  free
$W_G(H)$-spectra \cite{spcgq}  and this has the  model $\cA (G|(H))$
consisting of torsion modules over the
twisted group ring $H^*(BW_G^e(H))[W_G^d(H)]$ \cite{gfreeq2}, where
the Weyl group $W_G(H)=N_G(H)/H$ has identity component $W_G^e(H)$ and
discrete quotient $W_G^d(H)=\pi_0(W_G(H))$. We will
view this as stating that $\cA (G)$ consists of `sheaves' over the space
$\fX_G=\sub(G)/G$ of conjugacy classes of subgroups of $G$, but the precise
meaning of the word `sheaves' and the additional structure on these
`sheaves' needs considerable elucidation;  indeed it is the main content of
the model. 

In any case, this form suggests that if $H$ is a subgroup of $G$ and  we restrict sheaves over
$\sub(G)/G$ to conjugacy classes of subgroups of  $H$, we may
expect the category to be closely related to $\cA (H)$, though of
course the fusion of conjugacy classes along the map $\fX_H=\sub(H)/H\lra
\sub(G)/G=\fX_G$ will need to be taken into account, along with the
transition from $W_H(K)$ to $W_G(K)$. This in turn means that when
constructing $\cA (G)$ it is natural to adopt an inductive approach
and begin by giving a construction of $\cA (H)$ for all subgroups $H$
of $G$. Information on subgroups  of $U(2)$ is already in \cite{u2q},
which is in effect the most complicated part of the
model. Nonetheless, some work and  care is needed to assemble the
information for $SU(3)$ itself.  

\subsection{Partitions}  The fact that  $\sub (G)/G$ is the
Balmer spectrum of finite rational $G$-spectra \cite{spcgq} suggests some of the
relevant additional structure. The Balmer spectrum is equipped
with the Zariski topology, and we write $\fX_G=\sub (G)/G$ for this
space. As described in \cite{prismatic} we may use the language of
Priestley spaces, and state that $\fX_G$ has underlying constructible
topology on $\sub (G)/G$ being the h-topology (the quotient topology of the
Hausdorff metric topology on $\sub(G)$), and the spectral ordering is
the cotoral ordering\footnote{$K$ is {\em cotoral} in $H$ if $K$ is normal
  in $H$ with quotient a torus. }. Thus the closed sets of $\fX_G$ are precisely
the  $h$-closed sets
closed under cotoral specialization.

One may show in general that $\fX_G$ admits a partition into Zariski clopen blocks of
rather standard forms: thus
 $$\fX_G=\cV_1^G \amalg  \cV_2^G\amalg 
\cdots \amalg \cV_n^G,$$
 where each summand $\cV_i^G$ is `dominated' by a 
subgroup $H=H_i$ of $G$.  This means first of all that $W_G(H)$ is
finite. We may then choose an h-neighbourhood $\cN^H_H$ of $H$
in $\Phi (H)$ (conjugacy classes of subgroups of with finite Weyl
groups), so that  the image $\cN^G_H$ of $\cN^H_H$ in 
$\sub(G)/G$ by fusing $H$-conjugacy classes into $G$-conjugacy 
classes, also consists of subgroups with finite
Weyl group. Now we take  $\cV^G_i=\cV_H^G=\Lambda_{ct}(\cN^G_H)$ to be  the closure under 
cotoral specialization of $\cN^G_H$. In particular $\cV^G_H$ consists
of conjugacy classes of subgroups of $H$.

The existence of the partition of $\fX_G$ into blocks means that the algebraic
model splits as a product 
$$\cA (G)=\cA (G| \cV_1^G)\times \cA (G| \cV_2^G)\times \cdots \times
\cA (G| \cV_n^G). $$
This corresponds to the decomposition
$$\Gspectra \simeq \Gspectra \lr{\cV^G_1}\times
\Gspectra \lr{\cV^G_2}\times \cdots \times \Gspectra \lr{\cV^G_n}$$
into $G$-spectra with the specified geometric isotropy given by the
Burnside ring idempotents supported on the blocks.

Restricting attention to one piece, let us consider $\cA (G|\cV^G_H)$
where $\cV_H^G$ is the block dominated by the subgroup $H$. For
subgroups of $SU(3)$, it
will be easy to see what the category $\cA (G|\cV^G_H)$ should
be, but it is worth explaining how this should work more generally.

\subsection{General expectations}
The algebraic model $\cA (G|\cV^G_H)$  depends on the structure of the
subgroup $H$, on the fusion from
$H$-conjugacy to $G$-conjugacy and on the structure of normalizers. We
describe here the simplest possible behaviour; the general
behaviour is of a similar form. Suppose then that
the identity component of $H$ has the form $H_e=\Sigma \times_Z T$,
where $\Sigma$ is semisimple, $T$ is a torus and $Z$ is a finite
central subgroup. Here $T$ is the identity component of the centre
of $H_e$ so it is characteristic and acted upon by the finite component
group $W=H_d$. We consider the rational representation $\Lambda_0^{\Q}=H_1(T; \Q)$
of the finite group $W$ and write it as a sum of isotypical pieces
$$\Lambda_0^{\Q}=S_1^{n_1}\oplus S_2^{n_2}\oplus \cdots \oplus S_s^{n_s}$$
where the $S_i$ are pairwise non-isomorphic simple representations and
$S_1$ is the trivial representation. The simplest case is when this
rational decomposition comes from an integral 
decomposition of the toral lattice $\Lambda_0=H_1(T; \Z)$, and a
direct product decomposition of the torus
$$T=T_1\times T_2\times \cdots \times T_s. $$ 
In this simplest case
$$\cN^H_H=\{T_1\} \times \cN^H_2\times \cdots \times \cN^H_s, $$
where each $\cN^H_i$ is the compact, Hausdorff, totally disconnected
space  of $W$-invariant subgroups of $T_i$, and in the simplest case
this product decomposition is preserved by the fusion  to
$G$-conjugacy classes. In that case, up to fusion, $\cV^G_H$ is the product with
$\sub (T_1)$. 

The message is that the general form of $\fX_G$  is determined by
the isotypical decomposition of the rational toral lattice
$\Lambda_0^{\Q}$. Fusion and group theory mean that the actual
structure needs detailed analysis.

Almost all the examples $H$ that have so far been completely determined
have just one isotypical piece. The exception is the normalizer of the
maximal torus in $U(2)$ as in \cite{t2wqmixed}.

The model $\cA (G|\cV^G_H)$ takes the form of a sheaf of modules over a
sheaf of rings, with stalk $H^*(BW_G^e(K))$ over $K$. This has some
additional structure. First of all, the stalks over
cotoral subgroups are related in a way reflecting the Localization
Theorem, and the whole structure is equivariant for the component
structure given by the finite groups $W_G^d(K)$ associated to each subgroup.

\subsection{Associated work in preparation}
This paper is the fifth in a series  of 5 constructing an algebraic
category $\cA (SU(3))$ and showing it gives an algebraic model for 
rational $SU(3)$-spectra. This series gives a concrete 
illustrations of general results in small and accessible 
examples.

The first paper \cite{t2wqalg} describes the group theoretic data that feeds into the construction of an
abelian category $\cA (G)$ for a toral group $G$ and makes it
explicit for toral subgroups of rank 2 connected groups.

The second paper \cite{gq1} constructs algebraic models for all relevant 1-dimensional 
blocks, and the results are applied in the present
paper to each of the five 1-dimensional blocks for $SU(3)$.
The third paper \cite{t2wqmixed} constructs algebraic models for
blocks of rank 2 toral groups of mixed type.

The fourth paper \cite{u2q} constructs $\cA (U(2))$
in 7 blocks (5 of them  1-dimensional, and 2 of them 2-dimensional)
and shows it is equivalent to the category of rational
$U(2)$-spectra. The analysis of the 1-dimensional blocks uses
\cite{gq1} and the analysis of the block of the maximal torus
normalizer uses \cite{t2wqmixed}.

In this fifth paper, the work of the previous four parts is assembled
to a model of rational $SU(3)$-spectra. The 7 blocks from $U(2)$
survive with a little fusion, there are 2 more 1-dimensional blocks
and 9 new 0-dimensional blocks. 

This series is part of a more general programme. Future installments
will consider blocks with Noetherian Balmer spectra \cite{AGnoeth}
(the case where $\Lambda_0^{\Q}$ is a trivial representation)  and
those with no cotoral inclusions \cite{gqwf} (the case where
$\Lambda_0^{\Q}$ contains no trivial representation).
An account of the general nature of the models is in preparation
\cite{AVmodel}, and the author hopes that this will be the basis of the proof that the
category of rational $G$-spectra has an algebraic model in general.

\subsection{Organization}
In Section \ref{sec:proper} we  recall from \cite{gq1,u2q} 
the blocks $\cV^{G'}_H$ of $\fX_{G'} $ for various subgroups $G'$ of $G$.
In Section \ref{sec:oldsu3subgps} we explain the necessary changes to the
blocks $\cV^{G'}_H$ when taken up to $G$-conjugacy to form blocks
$\cV^G_H$ in $\fX_G$, and in Section \ref{sec:newsu3subgps} we
describe the blocks $\cV^G_H$ of groups $H$ not contained in $U(2)$.
Finally,  in Section \ref{sec:models} we describe the models $\cA
(G|\cV^G_H)$ where these differ from $\cA (H|\cV^{G'}_H)$ and explain
why they provide models for $G$-spectra over $\cV^G_H$. 

\subsection{Notation}
The ambient group throughout this paper is $G=SU(3)$, although we will
usually write the full name. The principal copy of $U(2)$ is the subgroup
preserving the decomposition $\C^3=\C^2\oplus \C$. We write $\T$ for
the subgroup of diagonal matrices: this is the chosen maximal torus of both
$U(2)$ and $SU(3)$. We write $Z$ for the centre of $U(2)$, which
consists of matrices $\diag (\lambda, \lambda, \lambda^{-2})$, and
 $\Tt$ for the  maximal torus of $SU(2)$ (consisting of matices $\diag
 (\lambda, \lambda^{-1}, 1)$).

Since $Z\cap \Tt $ is of order 2, we may need to consider central
products, and if $A\subseteq Z, B\subseteq SU(2)$ we write $A\times_2 B$
for the image of $A\times B$ in $U(2)=(Z\times SU(2))/C_2$ under the
central quotient. 

\section{Subgroups  of  proper subgroups of $SU(3)$ }
\label{sec:proper}

In constructing the models $\cA (G)$, it is convenient to proceed
group by group, steadily increasing the complexity of $G$. In fact $\cA (G)$  essentially
contains the models $\cA (H)$ for all subgroups $H$ of $G$, so it
is convenient to have a the models $\cA (H)$ to hand before tackling
$\cA (G)$. The word `essentially' covers two main changes: (i) several 
$H$-conjugacy classes may fuse to form a $G$-conjugacy class and (ii)
the normalizer of a subgroup $K$ in $H$ is a subgroup of the
normalizer in $G$ and hence the $H$-Weyl group $W_H(K)$ is a subgroup
of the $G$-Weyl group $W_G(K)$. These two changes mean that $\cA (H)$
will need to be adapted  to give the corresponding part of
the model in $G$, but the adjustments are secondary in nature. This
section gives a summary of the spaces $\fX_G$ for proper subgroups $G$
of $SU(3)$.

We may tabulate the dominant subgroups $H$ and their blocks as
follows. A subgroup $H$ is indicated by the pair $(H_e, F)$ where $F$ is a
finite subgroup of $W_G(H_e)$. In the rest of the section we will give
a little more detail. 
\label{sec:proper}
$$\begin{array}{ll|ccc|}
G&(H_e, F)&\dim (\cV^G_{(H_e,F)})&W_G(H)& \\
\hline 
{SO(2)}&(SO(2),1)&1&1&[8]\\
\hline 
O(2)&(SO(2),1)&1&C_2&[8]\\
&(SO(2),C_2)&1&1&[8]\\
\hline 
SO(3)&(T, 1) &1&C_2&[8]\\
&(T, C_2) &1&1&[8]\\
&(SO(3),1)&0&1&\mathrm{Discrete}\\
&(1, A_5)&0&1&\mathrm{Discrete}\\
&(1, \Sigma_4)&0&1&\mathrm{Discrete}\\
&(1, A_4)&0&C_2&\mathrm{Discrete}\\
&(1, D_4)&0&\Sigma_3&\mathrm{Discrete}\\
\hline 
U(2)&(T^2, 1) &2&C_2&[5,2]\\
&(T^2, C_2) &2&1&[9]\\
&(U(2),1)&1&1&[8]\\
&(Z, A_5)&1&1&[8]\\
&(Z, \Sigma_4)&1&1&[8]\\
&(Z, A_4)&1&C_2&[8]\\
&(Z, D_4)&1&\Sigma_3&[8]\\
\hline 
\end{array}$$

\subsection{The circle group $SO(2)$}
The proper subgroups of the circle group form the set $\cC$ of finite
cyclic groups, and $\fX_{SO(2)}=\sub (SO(2))=\cV^{SO(2)}_{SO(2)}$ is its one-point
compactification. The cotoral relation has finite subgroups cotoral in
$SO(2)$.  

Of course the group $Spin(2)$ is also a circle group, so we do not
need a separate entry. However we comment that factoring out the
elment of order 2 gives a map
$\fX_{Spin(2)}\lra \fX_{SO(2)}$ which is surjective,
but with fibres consisting of a single point over cyclic subgroups of even
order and two points over cyclic subgroups of odd order.

\subsection{The group $O(2)$}
The space $\fX_{O(2)}=\sub(O(2))/O(2)$ can be broken into  two blocks. The 
toral part $\cV^{O(2)}_{SO(2)}$  is 
equal to $\fX_{SO(2)}=\sub(SO(2))$ since each subgroup of
$SO(2)$ is characteristic. The remainder consists of the 
discrete space 
$\mcD$ of conjugacy classes of finite dihedral groups, together with 
$O(2)$ itself as the one point compactification. Altogether we have 
$$\fX_{O(2)}=\cV^{O(2)}_{O(2)}\amalg \cV^{O(2)}_{SO(2)}.$$

\subsection{The group $Pin(2)$}
The space $\fX_{Pin(2)}=\sub(Pin(2))/Pin(2)$ can be broken into  two blocks. The 
toral part $\cV^{Pin(2)}_{Spin(2)}$  is 
homeomorphic to $\fX_{Spin(2)}=\sub(Spin(2))$ since each
subgroup of $Spin(2)$ is characteristic. The remainder consists of the 
discrete space 
$\cQ$ of conjugacy classes of finite quaternion groups, together with 
$Pin(2)$ itself as the one point compactification. Altogether we have 
$$\fX_{Pin(2)}=\cV^{Pin(2)}_{Pin(2)}\amalg \cV^{Pin(2)}_{Spin(2)}.$$
The only reason for writing this out is to comment that although this is
homeomorphic to $\fX_{O(2)}$ the homeomorphism is factoring out the
centre on the $Pin(2)$ block, but not on the $Spin(2)$ block.

\subsection{The group $SO(3)$}
The space $\fX_{SO(3)}=\sub(SO(3))/SO(3)$ can be broken into 7 blocks. There
are 5 singleton blocks $\cV^{SO(3)}_{H}$ dominated by $H\in \{SO(3), A_5, \Sigma_4, A_4,
D_4 \}$. There is the block $\cV^{SO(3)}_{SO(2)}$ dominated by the
maximal torus $SO(2)$, and
there is the block $\cV^{SO(3)}_{O(2)}$ dominated by the
normalizer $O(2)$ of the maximal
torus. In summary, we have a partition  
$$\fX_{SO(3)}=\cV^{SO(3)}_{SO(2)}\amalg \cV^{SO(3)}_{O(2)}\amalg 
\cV^{SO(3)}_{SO(3)}\amalg  \cV^{SO(3)}_{A_5}\amalg \cV^{SO(3)}_{\Sigma_4}\amalg \cV^{SO(3)}_{A_4}  \amalg \cV^{SO(3)}_{D_4}  $$
into 7 clopen pieces. We note here that $\cV^{SO(3)}_{O(2)}$ is the
1-point compactification of $\mcD'$ of dihedral subgroups of order
$\geq 6$. The remaining two conjugacy classes from $\mcD$ are treated
separately. In $SO(3)$, the dihedral group $D_2$ is conjugate to $C_2$
so that it appears in $\cV^{SO(3)}_{SO(2)}$. It is convenient to treat
the dihedral group $D_4$ separately in $\cV^{SO(3)}_{D_4}$ because its Weyl group
is larger than that of all larger dihedral groups. 

\subsection{The group $SU(2)$}
At this level, the analysis of $SU(2)$ is exactly like that of
$SO(3)$. We need only replace the dominant groups $SO(3), O(2), SO(2), A_5, 
\Sigma_4, A_4$ and $D_4$ by their double covers,
$SU(2), Pin(2), Spin(2), \tilde{A}_5, 
\tilde{\Sigma}_4, \tilde{A}_4$ and $\tilde{D}_4$. The actual sets of
subgroups in each block are different of course.

However the changes are perhaps less
than expected, since finite subgroups outside the torus block
intersect the centre trivially. This means that reducing mod the centre gives a
homeomorphism $\cV^{U(2)}_{\tilde{H}}\cong \cV^{SO(3)}_H$ unless
$H=SO(2)$. This homeomorphism also preserves Weyl groups. On the other
hand, reduction modulo the centre gives a map
$\cV^{SU(2)}_{Spin(2)}\lra \cV^{SO(3)}_{SO(2)}$ which is surjective,
but with fibres consisting of a single point over cyclic subgroups of even
order and two points over cyclic subgroups of odd order. The Weyl
groups also  change as we discuss below. 

In any case, we have a partition  
$$\fX_{SU(2)}=\cV^{SU(2)}_{Spin(2)}\amalg \cV^{SU(2)}_{Pin(2)}\amalg 
\cV^{SU(2)}_{SU(2)}\amalg  \cV^{SU(2)}_{\tilde{A}_5}\amalg \cV^{SU(2)}_{\tilde{\Sigma}_4}\amalg \cV^{SU(2)}_{\tilde{A}_4}  \amalg \cV^{SU(2)}_{\tilde{D}_4}  $$

\subsection{The group $U(2)$}
The group $U(2)$ is quite involved. We give a summary here, and a full description is given in 
\cite{u2q}. 

The centre of $U(2)$ consists of scalar matrices and the quotient map 
$q: U(2)\lra PU(2)=SO(3)$ gives a bijection of blocks, so 
there are 7 blocks of $\sub (U(2))/U(2)$. Each of them is of 
dimension one more than the corresponding block in $SO(3)$. 

The partition of $SO(3)$ into the seven clopen pieces gives a clopen partition of $\fX_{U(2)}$ into seven blocks, 
$$p_*^{-1}\cV^{SO(3)}_H=\cV^{U(2)}_{p^{-1}H},  $$
where
$$p^{-1}(H)=\Ht\times_{C_2}Z, $$
where $Z=ZU(2)$ consists of the scalar matrices. 
Note that if $K$ lies in $p_*^{-1}(\cV^{SO(3)}_H)$ then $p(K)\subseteq (H)$ and so 
$\Kt=p^{-1}(K) \subseteq \Ht$ and $K\subseteq \Ht\times_{C_2}T$. 

Thus the partition takes the form 
$$\fX_{U(2)}=\cV^{U(2)}_{T^2}\amalg \cV^{U(2)}_{Pin(2)\times_{C_2}Z}\amalg 
\cV^{U(2)}_{U(2)}\amalg 
\cV^{U(2)}_{\tilde{A}_5\times_{C_2}Z}\amalg 
\cV^{U(2)}_{\tilde{\Sigma}_4\times_{C_2}Z}\amalg 
\cV^{U(2)}_{\tilde{A}_4 \times_{C_2}Z}\amalg 
\cV^{U(2)}_{\tilde{D}_4\times_{C_2}Z} ,  $$
where the first 2 blocks are 2-dimensional, and the remaining 5 blocks
are 1-dimensional. 

\section{Old subgroups of $SU(3)$}
\label{sec:oldsu3subgps}

Many subgroups of $SU(3)$ are conjugate to subgroups of $U(2)$, which 
were analyzed in \cite{u2q},  as summarised in Section 
\ref{sec:proper} above. For these subgroups, the only question is whether 
$U(2)$-conjugacy classes fuse in $SU(3)$. The question is settled in
this section. 

\subsection{Singular subgroups}
\label{subsec:singsub}
We make constant reference to subgroups of the maximal torus, so it is
worth collecting basic facts. The maximal torus consists of elements
$\diag (z_1, z_2, z_3) $ with $z_1z_2z_3=1$. The singular elements are
those lying in more than one maximal torus, and they are those lying
in the kernel $U_{\alpha}$  of some global root $\vartheta_{\alpha}: \T\lra T$,
which is to say they are elements with $z_i=z_j$ for some $i\neq
j$. The centre is the intersection of these: the cyclic group
generated by $\diag (\omega, \omega, \omega)$ with $\omega=e^{2\pi
  i/3}$.

We will repeatedly make the argument that the normalizer of a non-singular
subgroup normalizes the maximal torus.

\begin{lemma}
If $S\subseteq \T$ contains a non-singular element then
$N_G(S)\subseteq N_G(\T)$.
  \end{lemma}

  \begin{proof}
If $S\subseteq \T$ and $S^a=S$ then $S\subseteq \T^{a}$, so if $S$
contains a non-singular element then $\T^a=\T$.
    \end{proof}
    
This deals with most cases, and the remaining cases are simplified by
low dimensions and the following easily verified fact. 
  
\begin{lemma}
  \label{lem:singsub}
The only subgroups of $SU(3)$ consisting entirely of singular elements are those lying
in a single kernel $U_{\alpha}$.
\end{lemma}

Some other variants are also useful. 
\begin{lemma}
\label{lem:central}
If $a\not \in N_G(\T)$ and $A, B\subseteq \T$ with $A=B^a$ then $A$ lies in a singular circle. 
The only subgroups lying in two distinct central circles are the trivial group and the centre. 
\end{lemma}

\begin{proof}
Each maximal torus consists of the diagonal matrices for some 
orthonormal frame. Suppose $\T $ corresponds to the standard frame 
$E=(e_1, e_2, e_3)$ and $\T^a$ corresponds to $F=(f_1, f_2, f_3)$. Suppose 
$t$ has eigenvalues $(\lambda_1, \lambda_2, \lambda_3)$ for $E$. 

If we write $f_1=ae_1+be_2+ce_3$ then if $abc\neq 0 $, since $f_1$ is an eigenvalue of $t$, 
we have $\lambda_1=\lambda_2=\lambda_3$ and $t$ is central. If two of 
$a,b$ and $c$ are non-zero then the corresponding eigenvalues are equal and $t$ lies in 
a singular circle. The only remaining alternative that $f_1$ is a 
multiple of one of the $e_i$. Arguing similarly with $f_2, f_3$ we see 
that either $t$ is in a central circle or else the tori are equal, 
contradicting the assumption. 
\end{proof}

\subsection{Subgroups represented in $U(2)$}
\label{subsec:subutwo}
We are now equipped to consider the fusion on passage from $U(2)$ to
$SU(3)$. We show that in fact  there is there is 
very little new fusion. 

\begin{lemma}
  \label{lem:u2su3fusion}
The map $\sub(U(2))/U(2)\lra \sub(SU(3))/SU(3)$ is injective except for 
conjugacy classes dominated by the maximal torus, or the maximal
torus normalizer. In these cases the increase 
in Weyl group from $C_2$ to $\Sigma_3$  fuses subgroups in an easily 
understood way. 
\end{lemma}

\begin{proof}
 We have seen that $\fX_{U(2)} $ has 7 blocks. The images of 
 different blocks remain separate in 
 $\fX_{SU(3)}$, so we may consider them in turn. If $\cV$ is one 
 of the 5 blocks corresponding to  exceptional subgroups of 
 $SO(3)$ two subgroups which are isomorphic are already 
 conjugate. This proves injectivity on $\cV$. 

For the remainder we need to understand conjugacy of subgroups of the
maximal torus.  First, note that if $A, B\subseteq \T$ and $B^a=A$ then $A\subseteq \T\cap
\T^{a}$. If $\T^a=\T$ then $a\in N_G(\T)$ and the fusion is
standard. Otherwise $\T\cap \T^a$ is a proper subgroup of $\T$.
\end{proof}

Since subgroups of the circle are classified by their order and since
the three singular circles are conjugate, it follows that for the block of the maximal torus the only 
fusion is due to the enlargement of the Weyl group from $C_2$ to $\Sigma_3$. 

\begin{cor}
If $A,B\subseteq \T$ are conjugate in $SU(3)$ then $B^a$ is conjugate 
to $A$ in $U(2)$ where $a$ is a permutation matrix of order 1 or 3. 
\end{cor}

Finally, we consider the block of the maximal torus normalizer. The claim that there
is no new fusion is covered by the following. First we recall that any
full subgroup of the maximal torus normalizer in $U(2)$ is of the form
$H(A, \sigma)=\langle H, \sigma\rangle$,  where $A\subseteq \T$ and
$\sigma$ maps to the generator of the Weyl group. 

\begin{lemma}
  \label{lem:normfull}
Suppose given two full subgroups $H(A,\sigma)$, $H(B,\tau)$ of the
maximal torus normalizer. If $H(A,\sigma)=H(B,\tau)^a$ with $a\in
SU(3)$ then (i) $A=B^a$ (ii) $A=B^c$ where $c$ is a permutation matrix
of order 1 or 3 and (iii) $H(A,\sigma)=H(B, \tau)^d$ with $d\in U(2)$.  
\end{lemma}

\begin{proof}
We will argue below that if $H(A,\sigma)=H(B, \tau)^a$ then in fact
$A=B^a$.  From Lemma \ref{lem:central}  either $a\in N_G(\T)$ or else $A$ lies in a central circle. 
If $a\in N_G(\T)$ then the fusion is just what is expected from
passing from subgroups of $W_{U(2)}(\T)=C_2$ to
$N_{SU(3)}(\T)=\Sigma_3$.  However there is in fact no instance of a
subgroup $H(A,\sigma)$ so that a permutation matrix of order 3
preserves $A$: no rank 2 lattices $(12)$-invariant lattices of
$\Lambda^0$ are also invariant under $(123)$.

Finally we argue that if $A$ lies in a
central circle $H(A, \sigma)=H(A,\tau^a)$ implies conjugacy in $U(2)$.

In our case each of the full groups $H(A, \sigma)$ has $A$ as a
subgroup of index 2. If $A$ is not fixed by $\sigma$ then  is the nonidentity
block is determined by group theoretic properties. Indeed $C_H(x)$
is either (i) $H$ if $x\in A^W$ or (ii) $A$ if $x\in A\setminus A^W$
or (iii) $A^W$ if $x\not \in A$. 
Thus if  $H(B,\tau)^a=H(B^a, \tau^a)=H(A,\sigma)$ it follows  that
$B^a=A$. 

Finally if $A=A^W$ then $A\cap U(2)$ consists of scalar matrices, so
$A$ is generated by $\diag (\lambda, \lambda, \lambda^{-2})$ for some
$\lambda$. If $A$ is central then $A$ is of order 1 or 3, so $A=B^a$. 
Otherwise $A$ determines an orthogonal decomposition of $\C^3=V_1\oplus
V_2$, and hence $A$ is determined as the set of elements acting as a
scalar on $V_2$. Once again $A=B^a$.
\end{proof}

\section{New subgroups of $SU(3)$}
\label{sec:newsu3subgps}

We need to describe the subgroups not represented in $U(2)$. In
principle it might happen that such a subgroup $H$ might have a
cotoral subgroup inside $U(2)$, but this does not in fact happen.

\begin{prop}
There are 11 conjugacy classes of subgroups of $SU(3)$ 
not conjugate to subgroups of $U(2)$.

(i) $SU(3)$ itself

(ii) subgroups of local type $T\times T$ and block group
$\Sigma_3$ 
or $C_3$, together with their full subgroups. 

(iii) subgroups of local type $SU(2)$ isomorphic to
$SO(3)$ and

 (iv) subgroups of local type $1$, where there are 7 additional
 conjugacy classes of finite subgroups.
\end{prop}

In Subsection \ref{subsec:su3notu2} we will deduce this from the classical
Miller-Blichfeldt-Dickson classification of finite subgroups of
$SU(3)$. Theorem \ref{thm:su3blocks} gives a summary of the analysis.


\subsection{The strategy for classifying subgroups}
\newcommand{\fh}{\mathfrak{h}}
The general proof is constructive and we simply
follow it through in this case. It may be easier to follow if we describe the
general process first.

\begin{itemize}
\item {\bf Step 1:} Enumerate the local types $\fh$ that can
  occur. This is routine from the classification of Lie algebras. 

\item {\bf Step 2:} For each local type $\fh$, find which connected
  groups $H_e$ of local type $\fh$ can occur as subgroups of $G$, 
and classify occurrences up to conjugacy. If $G=U(n)$,
this is just a matter of counting isomorphism classes of
$n$-dimensional representations of $H_e$.  The well-developed
apparatus of representation theory makes this routine.

\item {\bf Step 3:} For each conjugacy class of connected subgroup 
$H_e$, identify conjugacy classes of subgroups $H$ with identity
component $H_e$. Since $H_e$ is normal in $H$ and  $N_G(H)\subseteq N_G(H_e)$ we may work 
inside $W_G(H_e)$. Accordingly, the problem is  precisely to classify
finite subgroups of $W_G(H_e)$ up to conjugacy. This can be  hard, but for 
small groups the classification is known. 
\end{itemize}
This process reduces us to consideration of the finite subgroups of
$W_G(H_e)$ for connected subgroups $H_e$. This includes finite
subgroups of $G$ (corresponding to $H_e=1$) but for $H_e\neq 1$ 
$W_G(H_e)$ is of smaller dimension than $G$.

We are left with a list of pairs $(H_e, F)$ with $H_e$ a connected
subgroup of $G$ and $F$ a finite subgroup of $W_G(H_e)$; the
associated subgroup $H$ is the inverse image of $F$ in $N_G(H_e)$.
Assuming that $\sub(G')/G'$ has already been identified for subgroups
$G'$ of $G$, we need only consider the maximal subgroups $(H_e,F)$.

\subsection{The strategy for decomposing into blocks}
We have described how to to identify all conjugacy classes of
subgroups, but what we really want to do is to decompose $\sub(G)/G$
into blocks. 

Starting with the largest groups, we allocate subgroups to
blocks. Suppose that the process has been started,  and that all
subgroups of higher rank  than the connected subgroup 
$H_e$ have been allocated to a block.

\begin{lemma}
Suppose $H_e$ is a  connected subgroup of $G$ and  write $W=W_G(H_e)$.
 Subgroups $H$ not included in previous blocks with identity component $H_e$
correspond to finite subgroups $F\subseteq W$ with $N_{W}(F)$ finite.
\end{lemma}

\begin{proof}
Any subgroup $H$ with identity component $H_e$ will lie in
$N_G(H_e)$, so the possibilities for $H$ correspond to finite
subgroups of the Weyl group $W_G(H_e)$. For each such finite subgroup
$F$ we may consider the inverse image $H$ in $N_G(H_e)$. If
$W_G(H)$ is infinite then $H$ is cotoral in a group considered
previously. Since $N_G(H)\subseteq N_G(H_e)$,
$N_G(H)/H_e=N_W(H/H_e)=N_W(F)$.
\end{proof}

\subsection{Conjugacy classes of maximal subgroups for $SU(3)$}
\label{subsec:su3notu2} 
The local types of subgroups of $SU(3)$ of rank 2 are $SU(3), 
SU(2)\times T$ and $T^2$. The local types of rank 1 are $SU(2)$ and 
$T$, and in rank 0 there is only the trivial local type.

{\em Local type $SU(3)$}: 
The group $G$ itself is the only subgroup of local type of $SU(3)$ and it is both 
h-isolated and cotorally isolated, giving 
$\cN^G_G=\cV^{G}_G=\{G\}$. 

{\em Local type $T^2$}:
 By the uniqueness of maximal tori, there is a single conjugacy class of connected subgroups of local type
$T^2$. It is well known that the Weyl group of the maximal torus is
$\Sigma_3$, and  the normalizer is a split extension $T^2\cdot
\Sigma_3$ (The group $\Sigma_3$ is generated by transpositions. We
lift these to the negative of the permutation matrices). 

Any subgroup $H$ of $N_G(T^2)$ with identity component $T^2$
 is of the form $\pi^{-1}(\Hbar)$ for a subgroup $\Hbar$ of
 $\Sigma_3$ where $\pi: N_G(T^2)\lra
\Sigma_3$ is the quotient map. This gives four components:
$$\cV^G_{T^2,1}, \cV^G_{T^2,C_2}, \cV^G_{T^2,C_3}, \cV^G_{T^2,\Sigma_3}$$

{\em Local type $SU(2)\times T$}: The simple
representations are of the form $V_i\tensor z^j$ where $V_i$ is the
$(i+1)$-dimensional representation of $SU(2)$ with weights $x^i,
x^{i-1}y, x^{i-2}y^2, \ldots, y^i$, and $z$ is the natural
representation of $T$. The representation $V_i\tensor z^j$ has
determinant $z^{(i+1)j}$, so it only lies in the special unitary group
if $j=0$, when it is not almost-faithful. The only options for an almost-faithful map to
$SU(3)$ are therefore $(V_1\tensor z^j)\oplus (V_0\tensor
z^{-2j})$. If  $j=0$ this is not almost-faithful, and if $|j|\geq
1$ then the representation factors through $SU(2)\times T/T[j]$.
Since we only care about the image, we only 
need to consider the case  $(V_1\tensor
z)\oplus (V_0\tensor z^{-2})$, giving the subgroup $U(2)=\{ A\oplus
(t) \st
A\in U(2), t=\det(A)^{-1}\}$. We note that $U(2)$ consists of all
matrices zero at $(1,3), (2, 3), (3,1), (3,2)$: 
$$U(2)=\left(
  \begin{array}{ccc}
    *&*&0\\
    *&*&0\\
    0&0&*
     \end{array}
   \right).$$
   and hence it it is self-normalizing with $W_G(U(2))=1$. This gives
   one further block $\cV^{G}_{U(2)}$. 

{\em Local type $T$}: 
 Connected subgroups of local type $T$ are conjugate to a subgroup of
 the maximal torus. There are three singular circles, all
of them conjugate; for definiteness, we focus on the subgroup $\diag (z, z,z^{-2})$
which is central in our chosen copy of $U(2)$.
One may check that the normalizer is also $U(2)$, so
the subgroups have already been seen in our analysis of $U(2)$. The other
circles are all regular, so their normalizers are the same as that of
the maximal torus, so these subgroups they have been seen in our
analysis of subgroups of $N_G(T^2)$.

{\em Local type $SU(2)$}:
Subgroups of local type $SU(2)$ have almost-faithful 3 dimensional
representations $V_1\oplus V_0$ and $V_2$. Those of the first type are
copies of $SU(2)$ (since $V_1$ is faithful) and those of the second
type are $SO(3)$ since $V_2$ has central elements in the kernel. One
checks  that $N_G(SU(2))=N_G(U(2))=U(2)$ and hence we have already seen
all subgroups of the first type. Finally, $SO(3)$ is the set of real
matrices in $SU(3)$, and therefore consists of points preserving a
real subspace of $\C^3$, so that $N_G(SO(3))=SO(3)$.

{\em Local type $1$}:
Finally there are finite subgroups not contained in any of the previous
groups.

Blichfeldt's classification \cite[Chapter 12]{BDM} of  finite
subgroups of $SU(3)$ has 5 families of 
classes denoted A, B, C, D, E. Type A consists of  abelian
subgroups, and all  of these lie  in a maximal torus, so have been
already considered above. Those with a two dimensional faithful representations are
subgroups of $U(2)$ and these have been dealt with above. Types
C and D occur insisde $T^2\sdr \Sigma_3$ and so have also been
considered above.



Finally, in Type E there are 7 exceptional finite subgroups
$PSL_2(7) , A_5\times C_3, PSL_2(7)\times C_3, A_6\cdot 3,
G_{36}\cdot 3 , G_{72}\cdot 3 , G_{216}\cdot 3$, giving 7 more
singleton blocks. For the general form of the model, the
 exact structure of these groups is not important; their Weyl groups
 are finite and the structure is read off  \cite{BDM} and the summary 
\cite{Ludl}.

\subsection{The partition for $SU(3)$} We are now in a position to
give a complete summary of the partition of closed subgroups of
$SU(3)$ up to conjugacy.

\begin{thm}
  \label{thm:su3blocks}
There is a partition of $\fX_G=\sub(G)/G$ into 18 Zariski clopen sets
$\cV^G_{H_e,F}$ as follows. We tabulate the pairs $(H_e, F)$ where
$H_e$ represents a conjugacy class of connected subgroups and $F$ is a
finite subgroup of $W_G(H_e)$. This is in order of decreasing size of
$(H_e,F)$. For each such pair we name the dimension of the
block $\cV^G_{H_e,F}$ dominated by $(H_e,F)$, the Weyl group of
$H$ (always finite). This much is intrinsic data. We then list the
subgroup where $\cV^G_{H_e,F}$ is first treated appropriately and the reference. 

$$\begin{array}{l|cccc|}
(H_e, F)&\dim (\cV^G_{(H_e,F)})&W_G(H)&\hat{H}& \\
\hline 
(G,1)&0&1&SU(3)&\mathrm{Discrete}\\
(U(2),1)&1&1&U(2)&[10]\\
(T^2, \Sigma_3) &1&1&SU(3)&[8]\\
(T^2, C_3) &1&C_2&SU(3)&[8]\\
(T^2, C_2) &2&1&U(2) &[10]\\
(T^2, C_1) &2&\Sigma_3&T^2&[5,2]\\
(SO(3),1)&0&1&SO(3)&\mathrm{Discrete}\\
(Z(U(2)), A_5)&1&1&U(2)&[8]\\
(Z(U(2)), \Sigma_4)&1&1&U(2)&[8]\\
(Z(U(2)), A_4)&1&C_2&U(2)&[8]\\
(Z(U(2)), D_4)&1&\Sigma_3&U(2)&[8]\\
(1,PSL_2(7))&0&C_3&SU(3)&\mathrm{Discrete}\\
(1,PSL_2(7)\times C_3)&0&1&SU(3)&\mathrm{Discrete}\\
(1,A_6\cdot 3)&0&1&SU(3)&\mathrm{Discrete}\\
(1,A_6)&0&C_3&SU(3)&\mathrm{Discrete}\\
(1,G_{36}\cdot 3)&0&1&SU(3)&\mathrm{Discrete}\\
(1,G_{72}\cdot 3)&0&1&SU(3)&\mathrm{Discrete}\\
(1,G_{216}\cdot 3)&0&1&SU(3)&\mathrm{Discrete}\\
\hline 
\end{array}$$
\end{thm}

\section{Models}
\label{sec:models}
Having described the partition of $\fX_{SU(3)}$ into 18 blocks, we may
describe the models $\cA (SU(3)|\cV^{SU(3)}_H)$, and explain where to
find the proofs that each does give a model.  

\subsection{Dimension 0} 
There are nine 0-dimensional singleton blocks. By definition their
Weyl group is finite, so the data for the
one dominated by $H$ is just the finite group $W_G(H)$.
 It is shown in
\cite{gfreeq2} that $\Q[W_G(H)]$-modules give a model.

Seven are finite  groups occurring first for $SU(3)$. The other two
are $SO(3)$ and $SU(3)$ itself, both of which are self-normalizing. 

\subsection{Dimension 1} 
There are seven 1-dimensional blocks. The data for these consists
of a sheaf of rings and a component structure. 

The blocks dominated by $(T^2, \Sigma_3)$ and $(T^2, C_3)$ occur
first for $SU(3)$. As blocks of the corresponding toral group they are
of Type 0, so there is nothing to say
about the sheaf of rings. The component structure in the ambient toral
group is determined in \cite{t2wqalg}. However the sheaf of rings and
component structure is a little different in $G=SU(3)$ itself. Indeed,
some of the groups (eg the element of order 3) have infinite Weyl groups in the whole
groups. Similarly, even subgroups with finite Weyl groups in the
ambient toral group could in principle  have a bigger Weyl group in
$SU(3)$.

\begin{lemma}
  Let $H=T^2\sdr C_3$ or $T^2\sdr \Sigma_3$ and $G=SU(3)$
  and suppose $K$ is a full subgroup of $H$ with toral part $S$.

 (a)  If $S$ is not singular $N_G(K)=N_H(K)$.

 (b) If $S$ is singular then $S$ is central, then $K$ is
 self-normalizing in $H$, so that  $W_H(K)=1$.
 If $H=T^2\sdr C_3$ then  $W_G(K)$ is a 2-torus, and
 if $H=T^2\sdr \Sigma_3$ then  $W_G(K)$ is of order 2. 
\end{lemma}

\begin{proof}
  Suppose $K=H(S,\sigma)$ and $a\in N_G(K)$. Then by Lemma
  \ref{lem:normfull} we find $a\in N_G(S)$. If $S$ is not singular in $G$ then
  this means $a\in N_G(T^2)$.  Since $S$ is invariant under $C_3$ in
  both cases, this means that $S$ is central, and therefore $S=1$ or
  $S=C_3$.

  If $S=1$ and $H=T^2\sdr C_3$ then $K$ is a non-central subgroup of order
  $3$. It lies in some (other) maximal torus $T'$, where it is
  non-singular (the component group of $H$ permutes the three
  coordinates transitively). Its normalizer therefore lies in $N_G(T')$, and it is
  not invariant under any non-identity element of $W_G(T')$. Its
  normalizer is thus $T'$.  If $H=T^2\sdr \Sigma_3$ then may argue
  similarly with the Sylow 3-subgroup, so that $K=H(C_3, \tau)$ for an
  involution $\tau$ and $W_G(K)$ is of order 2. 

  If $S=C_3$ is central we may argue similarly in $G/Z=PSU(3)$.
   \end{proof}

The five blocks dominated by a subgroup of $U(2)$ are of Type 1 in
$U(2)$. The one dominated by $U(2)$ is straightforward.

\begin{lemma}
Any subgroup $K$ in the block dominated by $U(2)$  has the same normalizer in $SU(3)$
as in $U(2)$.
\end{lemma}

\begin{proof}
  The subgroup $U(2)$ itself is self-normalizing.    Any other
  subgroup in this block has identity component
  $SU(2)$,  which is also self-normalizing. Since $N_G(K)\subseteq
  N_G(K_e)$, this completes the proof.
  \end{proof}

  The remaining four need individual attention.  For notation,
  we suppose $A\in \{ A_5, \Sigma_4, A_4, D_4\}$, $\At$ is the double cover 
    of $A$ and $\Ah=Z\At=\At\times_2 Z$. 

    \begin{lemma}
      If $K$ is dominated by $\Ah$ in $U(2)$ (i.e., $K \in
      \cV^{U(2)}_{\Ah}$ ) then $N_{SU(3)}(K)=N_{U(2)}(K)$. 
        \end{lemma}

        \begin{proof}
If $a$ is the central element of order 2 in $U(2)$  (ie $a=-I$) then
$\langle a\rangle $ lies in $K$ and is characteristic in $K$. Thus if
$g\in N_{SU(3)}(K)$ we have $a^g=a$ and $g\in C_G(a)$. By the standard
formula for centralizers (\cite[V.2.3]{BtD}) $C_G(a)=U(2)$.
          \end{proof}

Now that  we have determined the sheaf of rings and component structure, we
may turn to models.  We proved in \cite{gq1} that each 
1-dimensional block $\cV$ with a single height 1 point and all other
points being finite has $\cA(G|\cV)$ as a model. This applies to all
1-dimensional blocks except for the $(U(2),1)$ block. 

Furthermore, we proved in Lemma \ref{lem:u2su3fusion} that the data
for the model  is unchanged for all relevant subgroups of $U(2)$. Thus
inclusion $U(2)\lra SU(3)$ induces an equivalence for these
blocks. 

The block $(U(2), 1)$ is well behaved: all subgroups are normal
and cotoral in $U(2)$ to the methods of \cite{gq1} apply once again to
show
\begin{enumerate}
\item 
   $\cA (SU(3)|\cV^{SU(3)}_{U(2)})$ is a model for rational
$SU(3)$-spectra with geometric isotropy dominated by $U(2)$, 
\item   $\cA (U(2)|\cV^{U(2)}_{U(2)})$ is a model for rational
$U(2)$-spectra with geometric isotropy dominated by $U(2)$, and 
\item  restriction from $SU(3)$ to $U(2)$ induces an isomorphism of 
categories on this block.  
\end{enumerate}

\subsection{Dimension 2} 
There are two 2-dimensional blocks. The one dominated by the
maximal torus $\cA (G|\mathrm{toral})$  is treated in \cite{AGtoral,
  gtoralq}. However there is quite a lot to spell out. The Weyl groups
of individual abelian subgroups $A$ are as follows.

We recall that in $SU(3)$ the maximal torus $T^2$ consists of diagonal
matrices $\diag (z_1, z_2, z_3)$ with $z_1z_2z_3=1$. The singular
elements are those with some pair of $z_i$ being equal, and the centre
consists of those where all three are equal. 

\begin{lemma}
  \label{lem:normtorus}
  Suppose  $S\subseteq T^2$. 

  (a) If $S$ is non-singular then $N_G(S)\subseteq T^2\sdr \Sigma_3$, 
  and is generated by $T^2$ together with the permutations preserving 
  it. 

  (b) If $S$ is singular then (after conjugating if necessary) it lies 
  in the circle with $z_1=z_2$. If it is not central, then $N_G(S)=N_{U(2)}(S)$. 

  (c) If $S$ is central then $N_G(S)=G$. 
\end{lemma}

\begin{proof}
(a) This argument is becoming familiar. If $S$ is non singular and
$S=S^a$  then $T=T^a$.

(b) By Lemma \ref{lem:singsub} singular subgroups must be of the stated
form. 

Now if $S$ consists of elements with $z_1=z_2$ and there is an element
with $z_3\neq z_1$ we find $S$ preserves the decomposition $\C^2\oplus
\C$, and therefore any element normalizing it does too, but this characterizes $U(2)$.
 \end{proof}

The block dominated by $T^2\sdr C_2$ in the toral case  is treated in \cite{u2q}. 
We need detailed discussion of the Weyl groups in $U(2)$ and $SU(3)$.

\begin{lemma}
  Suppose  $K$ is a full subgroup of $T^2\sdr C_2$, and $S=K\cap T^2$.  

  (a) If $S$ is non-singular then $N_{SU(3)}(K)\subseteq T^2\sdr \Sigma_3$, 
  and is generated by $T^2$ together with the permutations preserving 
  $S$. 

  (b) If $S$ is singular then (after conjugating if necessary) it lies 
  in the circle with $z_1=z_2$. If it is not central, then $N_{SU(3)}(S)=N_{U(2)}(S)$. 

  (c) If $S$ is central then
$N_{SU(3)}(S)=SU(3)$. 
\end{lemma}

\begin{proof}
To start with Lemma \ref{lem:normfull} (i) shows that the normalizer
preserves $S$. Parts (a) and (b) follow by the same argument as for
Lemma \ref{lem:normtorus}.

The remaining cases have $K=C_2$ or $K=C_3\times C_2$. In either 
case, the involution in $U(2)$ is conjugate to $\diag (1, -1)$, which
embeds as $\diag (1, -1, -1)$ in $SU(3)$. This picks out a copy of
$\C^2$.
  \end{proof}

The fact that there are only toral
inclusions are of codimension  1 means that there is a model based on
pairs without using flags. Furthermore the lack of any new fusion
means again that  restriction from $SU(3)$ to $U(2)$ induces an isomorphism of 
categories on this block, so that
$$\cA (SU(3)|\cV^{SU(3)}_{(\T, C_2)})=\cA (U(2)|\cV^{U(2)}_{(\T, C_2)}) =\cA (\T\sdr
C_2|\cV^{\T\sdr C_2}_{(\T, C_2)\geq 3}), $$
where we have emphasized in the last expression that only anticentral
subgroups of order $\geq 3$ are to be used.

\bibliographystyle{plain}
\bibliography{../../jpcgbib}
\end{document}